\documentclass{article}
\usepackage{graphicx,indentfirst}
\usepackage{amscd}

\usepackage[all]{xy}
\usepackage{amsmath}
\usepackage{amsfonts,amssymb}
\usepackage{extarrows}
\usepackage{amsthm}
\usepackage{latexsym}
\usepackage[american]{babel}
\usepackage{times}
\usepackage{tikz}
\usepackage{tikz-cd}
\usepackage{url}
\usepackage{verbatim}
\usepackage{enumerate}
\usepackage{hyperref}

\usetikzlibrary{matrix,arrows,decorations.pathmorphing}
\usepackage[a4paper,left=2.5cm,right=2.5cm,bottom=3cm,top=3.5cm]{geometry}

\makeatletter
\newcommand*{\relrelbarsep}{.386ex}
\newcommand*{\relrelbar}{%
  \mathrel{%
    \mathpalette\@relrelbar\relrelbarsep
  }%
}
\newcommand*{\@relrelbar}[2]{%
  \raise#2\hbox to 0pt{$\m@th#1\relbar$\hss}%
  \lower#2\hbox{$\m@th#1\relbar$}%
}
\providecommand*{\rightrightarrowsfill@}{%
  \arrowfill@\relrelbar\relrelbar\rightrightarrows
}
\providecommand*{\leftleftarrowsfill@}{%
  \arrowfill@\leftleftarrows\relrelbar\relrelbar
}
\providecommand*{\xrightrightarrows}[2][]{%
  \ext@arrow 0359\rightrightarrowsfill@{#1}{#2}%
}
\providecommand*{\xleftleftarrows}[2][]{%
  \ext@arrow 3095\leftleftarrowsfill@{#1}{#2}%
}
\makeatother

\RequirePackage[T1]{fontenc}

\DeclareMathOperator*{\holim}{holim}

\begin{document}

\title{Twisted complexes and simplicial homotopies}
\author{Zhaoting Wei\footnote{Department of Mathematics, Texas A\&M  University-Commerce, Commerce, TX, 75429, 
\href{zhaoting.wei@tamuc.edu}{zhaoting.wei@tamuc.edu}}}

\newcommand{\End}{\text{End}}
\newcommand{\ad}{\text{ad}}
\newcommand{\tr}{\text{tr}}
\newcommand{\bDelta}{\bf{\Delta}}
\newcommand{\Hom}{\text{Hom}}
\newcommand{\D}{\text{D}}
\newcommand{\Tot}{\text{Tot}}
\newcommand{\dgCat}{\text{dgCat}}
\newcommand{\DK}{\text{DK}}
\newcommand{\cat}{\mathcal}
\newcommand{\obj}{\text{obj}}
\newcommand{\id}{\text{id}}
\newcommand{\op}{\text{op}}
\newcommand{\Tw}{\text{Tw}}
\newcommand{\Perf}{\text{Perf}}
\newcommand{\StrPerf}{\text{StrPerf}}
\newcommand{\Cpx}{\text{Cpx}}

\newtheorem{thm}{Theorem}[section]
\newtheorem{lemma}[thm]{Lemma}
\newtheorem{prop}[thm]{Proposition}
\newtheorem{coro}[thm]{Corollary}
\newtheorem{conj}[thm]{Conjecture}
\theoremstyle{definition}\newtheorem{defi}[thm]{Definition}
\theoremstyle{remark}\newtheorem{eg}{Example}[section]
\theoremstyle{remark}\newtheorem{rmk}{Remark}[section]
\theoremstyle{remark}\newtheorem{ctn}{Caution}

\maketitle

\setcounter{section}{-1}

\begin{abstract}
In this paper we consider the dg-category of twisted complexes over simplicial ringed spaces. It is clear that a simplicial map $f: (\mathcal{U},\mathcal{R})\to (\mathcal{V}, \mathcal{S})$ between simplicial ringed spaces induces a dg-functor $f^*: \Tw(\mathcal{V}, \mathcal{S})\to \Tw(\mathcal{U}, \mathcal{R})$ where $\Tw(\mathcal{U}, \mathcal{R})$ denotes the dg-category of twisted complexes on $(\mathcal{U},\mathcal{R})$. In this paper we prove that for simplicial homotopic maps $f$ and $g$, there exists an $A_{\infty}$-natural transformation $\Phi: f^*\Rightarrow g^*$ between induced dg-functors. Moreover the $0$th component of $\Phi$ is an objectwise weak equivalence. If we restrict ourselves to the full dg-subcategory of twisted perfect complexes, then we prove that  $\Phi$ admits an $A_{\infty}$-quasi-inverse when $(\mathcal{U},\mathcal{R})$ satisfies some additional conditions.
\end{abstract}
\section*{}
MSC: 18D20, 18G55, 18G30, 14F05

\section{Introduction}
In the late 1970's Toledo and Tong \cite{toledo1978duality} introduced twisted complexes as a way to get their hands on perfect complexes of sheaves on a complex manifold. Twisted complexes, which consist of locally defined complexes together with higher transition functions, soon play an important role in the study of complex geometry, algebraic geometry, as well as dg-categories and $A_{\infty}$-categories, see \cite{o1981trace}, \cite{o1981hirzebruch}, \cite{o1985grothendieck}, \cite{bondal1990enhanced}, \cite{wei2016twisted}, \cite{block2017explicit}, \cite{tsygan2013microlocal}, and \cite{arkhipov2018homotopy2}.

In particular, in  \cite{block2017explicit} and \cite{arkhipov2018homotopy2} it has been proved that for a simplicial ringed space $ (\mathcal{U},\mathcal{R})$, the dg-category of twisted complexes $\Tw(\mathcal{U}, \mathcal{R})$ (See Definition \ref{defi: twisted complexes on simplicial spaces} below) gives the homotopy limit of the cosimplicial diagram of dg-categories
\begin{equation}\label{equation: second cosimplicial diagram of Cpx of an open cover in the introduction}
\begin{tikzcd}
\Cpx(U_0, \mathcal{R}_0) \arrow[yshift=0.7ex]{r}\arrow[yshift=-0.7ex]{r}& \Cpx(U_1, \mathcal{R}_1) \arrow[yshift=1ex]{r}\arrow{r}\arrow[yshift=-1ex]{r}  &   \Cpx(U_2, \mathcal{R}_2) \cdots
\end{tikzcd}
\end{equation}
where $\Cpx(U_i, \mathcal{R}_i) $ denotes the dg-category of complexes of sheaves of $\mathcal{R}_i$-modules on $U_i$. See Proposition \ref{prop: twisted complex is homotopy limit} below.

\begin{rmk}
The definition of twisted complexes in this paper is slightly different to  twisted complexes introduced in \cite{bondal1990enhanced}. See  Definition \ref{defi: twisted complexes on simplicial spaces} below and \cite[Definition 1]{bondal1990enhanced}.
\end{rmk}

Therefore it is natural to expect that the dg-category $\Tw(\mathcal{U}, \mathcal{R})$ has some kind of homotopy invariance. In particular let $f$ and $g:  (\mathcal{U},\mathcal{R})\to (\mathcal{V}, \mathcal{S})$ be two simplicial maps which are simplicial homotopic, i.e. there exists a simplicial map
$$
H: \mathcal{U}\times I \to \mathcal{V}
$$
such that $f=H\circ \varepsilon_0$ and $g=H\circ \varepsilon_1$, we expect that the induced dg-functors $f^*$ and $g^*: \Tw(\mathcal{V}, \mathcal{S})\to \Tw(\mathcal{U}, \mathcal{R})$ can be identified.

Using $H$ we can construct, for each object $\mathcal{E}\in \Tw(\mathcal{V}, \mathcal{S})$, a degree $0$ morphism
$$
\Phi_0(\mathcal{E}): f^*(\mathcal{E})\to g^*(\mathcal{E}).
$$
In Proposition \ref{prop: homotopic functors induces quasi-isomorphic dg-functors between twisted complexes} we prove that $\Phi_0(\mathcal{E})$ is closed and in addition a weak equivalence for each $\mathcal{E}$.

Unfortunately, for a morphism $\phi: \mathcal{E}\to \mathcal{F}$ we notice that
$$g^*(\phi) \cdot\Phi_{0}(\mathcal{E})-(-1)^{|\phi|} \Phi_{0}(\mathcal{F})\cdot f^*(\phi)\neq 0.$$
Therefore $\Phi_{0}$ does not give a dg-natural transformation from $f^*$ to $g^*$. Nevertheless, in this paper we extend $\Phi_0$ to an $A_{\infty}$-natural transformation $\Phi: f^* \Rightarrow g^*$, see Theorem \ref{thm: simplicial homotopies and A infinity nt} below. In addition, if we restrict to $\Tw_{\Perf}(\mathcal{V}, \mathcal{S})$, the full dg-subcategory of twisted perfect complexes, then we can show that $\Phi$ has an $A_{\infty}$-quasi-inverse.

This paper is organized as follows: in Section \ref{section: twisted complexes} we review the concept of twisted complexes and in Section \ref{section: A infinity nt} we review $A_{\infty}$-natural transformations between dg-functors. In Section \ref{section: Simplicial homotopies and twisted complexes} we first study simplicial homotopies between simplicial maps and then construct the $A_{\infty}$-natural transformation $\Phi$. In Section \ref{section: simplicial homotopy and twisted perfect complexes} we consider twisted perfect complexes and show that in this case the $A_{\infty}$-natural transformation $\Phi$ admits an the $A_{\infty}$-quasi-inverse  if $(\mathcal{U},\mathcal{R})$ satisfies some additional conditions.

\section*{Acknowledgement}
The author wants to thank Julian V.S. Holstein for very helpful discussions.

\section{A review of twisted complexes}\label{section: twisted complexes}
\subsection{A review of simplicial and cosimplicial objects}
Recall that the simplicial category $\Delta$ is the category with objects
$$
[n]=\{0,\ldots,n\} ~\text{ for } n\geq 0
$$
and morphisms order preserving functions between objects.

Let $\mathcal{C}$ be a category. A \emph{simplicial object} $X$ in $\mathcal{C}$ is a contravariant functor
$$
X: \Delta^{\op}\to \mathcal{C}
$$
and a morphism $f: X\to Y$ between two simplicial objects in $\mathcal{C}$ is a natural transformation between contravariant functors. 

More explicitly, a simplicial object $X$ in $\mathcal{C}$ consists of a collection of objects $X_n\in \obj\mathcal{C}$ for $n\geq 0$ and a collection of face morphisms
$$
\partial_i: X_n\to X_{n-1}
$$
and degeneracy morphisms
$$
 s_i: X_n\to X_{n+1}  ~0\leq i\leq n
$$
which satisfy the \emph{simplicial identities}
\begin{equation}\label{eq: simplicial identities}
\begin{split}
\partial_i\partial_j&=\partial_{j-1}\partial_i ~\text{ if } i<j;\\
\partial_i s_j&= s_{j-1}\partial_i ~\text{ if } i<j;\\
\partial_i s_j&= \id ~\text{ if } i=j \text{ or } i=j+1;\\
 \partial_i s_j&= s_j\partial_{i-1}~\text{ if } i>j+1;\\
s_is_j&=s_{j+1}s_i ~\text{ if } i\leq j.
\end{split}
\end{equation}
A morphism $f: X\to Y$ between two simplicial objects consists of a collection of morphisms $f_n: X_n\to Y_n$ for $n\geq 0$ in $\mathcal{C}$ such that $f_n$ is compatible with all $\partial_i$'s and $s_j$'s.

Dually a \emph{cosimplicial object} $X$ in $\mathcal{C}$ is a covariant functor
$$
X: \Delta\to \mathcal{C}
$$
and a morphism $f: X\to Y$ between two cosimplicial objects in $\mathcal{C}$ is a natural transformation between covariant functors. We also have an explicit description of cosimplicial objects and morphisms which is dual to the simplicial case.

\begin{eg}\label{eg: classifying space of an open cover}[Classifying space of an open cover]
Let $X$ be a topological space and let  $\mathcal{U}=\{U_i\}$ be an open cover of $X$. Let $U_{i_0\ldots i_n}$ denote the intersection $U_{i_0}\cap\ldots \cap U_{i_n}$ where repetitions of indices are allowed. Then we get a simplicial space $\mathcal{N}$ where
$$
N_n=\coprod_{i_0,\ldots,i_n}U_{i_0\ldots i_n}.
$$
The face map $\partial_k: N_n\to N_{n-1}$ is induced by the inclusion map
$$
U_{i_0\ldots i_n} \hookrightarrow U_{i_0\ldots \widehat{i_k}\ldots i_n}
$$
and the degeneracy map $s_k: N_n\to N_{n+1}$ is induced by the identity map
$$
U_{i_0\ldots i_n}\overset{=}{\to}U_{i_0\ldots i_ki_k\ldots i_n}
$$
\end{eg}

 \subsection{Notations of bicomplexes and sign conventions}\label{subsection: notation of bicomplexes}
In this section by a ringed space we mean a topological space $X$ together with a sheaf of (not necessarily commutative) rings $\mathcal{R}$ on $X$. Examples include
\begin{itemize}
\item A scheme $X$ with the structure sheaf $\mathcal{O}_X$;
\item A complex manifold $X$ with the  sheaf of analytic functions $\mathcal{O}_X$;
\item A topological space $X$ with the constant sheaf of rings $\underline{\mathbb{C}}$;
\item A scheme $X$ with the sheaf of  rings of differential operators $\mathcal{D}_X$.
\end{itemize}

\begin{rmk}
\cite{toledo1978duality}, \cite{o1981trace} and \cite{o1985grothendieck} focus on the special case that $X$ is a complex manifold and $\mathcal{R}=\mathcal{O}_X$ is the sheaf of holomorphic functions on $X$. In this paper we consider more  general  $(X,\mathcal{R})$.
\end{rmk}

\begin{rmk}
In this section by $\mathcal{R}$-modules we always mean left $\mathcal{R}$-modules, unless it is explicitly pointed out otherwise.
\end{rmk}

A \emph{simplicial ringed space} is a simplicial object in the category of ringed spaces, and a \emph{simplicial map} is a morphism between two simplicial objects  in the category of ringed spaces.

In this section we introduce some notations which are necessary in the definition of twisted complexes.
Let $(\mathcal{U},\mathcal{R})$ be a simplicial ringed space. 
Let $\partial_i: (U_n, \mathcal{R}_n)\to (U_{n-1}, \mathcal{R}_{n-1})$ and $s_i: (U_n, \mathcal{R}_n)\to (U_{n+1}, \mathcal{R}_{n+1})$ be the face and degeneracy maps, respectively. Moreover for $k\geq p$, we define
$
 \rho_{k,p}: (U_k, \mathcal{R}_k)\to (U_p, \mathcal{R}_p)$
to be the front face map, i.e. 
\begin{equation}\label{eq: rho}
\rho_{k,p}:=\partial_{p+1}\circ\partial_{p+2}\circ\ldots\circ\partial_k.
\end{equation} 
Similarly we define $
 \tau_{k,p}: (U_k, \mathcal{R}_k)\to (U_p, \mathcal{R}_p)$
 to be the back face map, i.e. 
\begin{equation}\label{eq: tau}
\tau_{k,p}:=\partial_0\circ\partial_0\circ\ldots\circ\partial_0.
\end{equation}

We have the following identities.

\begin{lemma}\label{lemma: rho tau f}
For $k\geq p\geq r$ we have 
\begin{equation}
\rho_{p,r}\circ \rho_{k,p}=\rho_{k,r},
\end{equation}
\begin{equation}
\tau_{p,r}\circ \tau_{k,p}=\tau_{k,r},
\end{equation}
\begin{equation}
\rho_{p,r}\circ \tau_{k,p}=\tau_{k+r-p,r}\circ \rho_{k,k+r-p}.
\end{equation}
Moreover, for a morphism $f: (X, \mathcal{R})\to (Y,\mathcal{T})$ between simplicial ringed spaces, we have
\begin{equation}
f_p\circ \rho_{k,p}= \rho_{k,p}\circ f_k, ~ f_p\circ \tau_{k,p}= \tau_{k,p}\circ f_k.
\end{equation}
\end{lemma}
\begin{proof}
It follows from the simplicial identities \eqref{eq: simplicial identities}
\end{proof}

Let $E^{\bullet}$ be a   graded sheaf of $\mathcal{R}_0$-modules on $U_0$. Let
\begin{equation}\label{equation: bigrade sheaves}
C^{\bullet}(\mathcal{U},\mathcal{R}, E^{\bullet})=\prod_{p,q}\Gamma(U_p,\rho_{p,0}^*E^q)
\end{equation}
be the bigraded complexes of $E^{\bullet}$. 

Now if another   graded sheaf $F^{\bullet}$ of $\mathcal{R}_0$-modules is given on  $U_0$, then we can consider the bigraded complex
\begin{equation}\label{equation: map with bigrade between graded sheaves}
C^{\bullet}(\mathcal{U},\mathcal{R}, \text{Hom}^{\bullet}(E,F))=\prod_{p,q} \text{Hom}^q_{\mathcal{R}_p-\text{Mod}}(\tau_{p,0}^*E,\rho_{p,0}^*F).
\end{equation}

\begin{rmk}\label{remark: (p,q) Cech and cohomology degree}
In this paper when we talk about degree $(p,q)$, the first index always indicates the simplicial degree while the second index always indicates the graded sheaf degree. We use $|u|$ to denote  the total degree of $u$.
\end{rmk}

We need to study the compositions of $C^{\bullet}(\mathcal{U},\mathcal{R}, \text{Hom}^{\bullet}(E,F))$. Let $G^{\bullet}$ be a third   graded sheaf of $\mathcal{R}_0$-modules, then there is a composition map
$$
C^{\bullet}(\mathcal{U},\mathcal{R}, \text{Hom}^{\bullet}(F,G)) \times C^{\bullet}(\mathcal{U},\mathcal{R}, \text{Hom}^{\bullet}(E,F))\to C^{\bullet}(\mathcal{U},\mathcal{R}, \text{Hom}^{\bullet}(E,G)).
$$
In fact, for $u^{p,q}\in C^p(\mathcal{U},\mathcal{R}, \text{Hom}^{q}(F,G))$ and $v^{r,s} \in C^r(\mathcal{U},\mathcal{R}, \text{Hom}^s(E,F))$, their composition $(u\cdot v)^{p+r,q+s}$ is given by
\begin{equation}\label{equation: composition of maps between graded sheaves}
(u\cdot v)^{p+r,q+s} =(-1)^{qr}(\rho_{p+r,p}^*u^{p,q})\circ (\tau_{p+r,r}^*v^{r,s})
\end{equation}
where the right hand side is the na\"{i}ve composition of sheaf maps.

In particular $C^{\bullet}(\mathcal{U},\mathcal{R}, \text{Hom}^{\bullet}(E,E))$ becomes an associative algebra under this composition (It is easy but tedious to check the associativity). We also notice that $C^{\bullet}(\mathcal{U},\mathcal{R}, E^{\bullet})$ becomes a left module over this algebra. In fact the action
$$
C^{\bullet}(\mathcal{U},\mathcal{R}, \text{Hom}^{\bullet}(E,E))\times C^{\bullet}(\mathcal{U},\mathcal{R}, E^{\bullet})\to  C^{\bullet}(\mathcal{U},\mathcal{R}, E^{\bullet})
$$
 is given by
\begin{equation}\label{equation: action of maps on sheaves}
(u\cdot c)^{p+r,q+s} =(-1)^{qr}(\rho_{p+r,p}^*u^{p,q})\circ (\tau_{p+r,r}^*c^{r,s})
\end{equation}
where the right hand side is given by evaluation.

\begin{rmk}
The definition of compositions and actions makes sense because we have Lemma \ref{lemma: rho tau f}.
\end{rmk}

There is also  a \v{C}ech-style differential operator $\delta$ on $C^{\bullet}(\mathcal{U},\mathcal{R}, \text{Hom}^{\bullet}(E,F))$ and $C^{\bullet}(\mathcal{U},\mathcal{R}, E^{\bullet})$ of bidegree $(1,0)$ given by the formula
\begin{equation}\label{equation: delta on maps}
(\delta u)^{p+1,q}=\sum_{k=1}^p(-1)^k \partial_k^*u^{p,q}  \,\text{ for } u^{p,q}\in C^p(\mathcal{U},\mathcal{R}, \text{Hom}^q(E,F))
\end{equation}
and
\begin{equation}\label{equation: delta on sheaves}
(\delta c)^{p+1,q} =\sum_{k=1}^{p+1}(-1)^k \partial_k^* c^{p,q}  \,\text{ for }c^{p,q}\in C^p(\mathcal{U}, \mathcal{R},E^q).
\end{equation}

\begin{ctn}
Notice that the map $\delta$ defined above is different from the usual \v{C}ech differential. In Equation \eqref{equation: delta on maps} we do not include the $0$th and the $(p+1)$th indices and in Equation \eqref{equation: delta on sheaves} we do not include the $0$th  index.
\end{ctn}

\begin{prop}\label{prop: Leibniz for Cech differential}
The  differential satisfies the Leibniz rule. More precisely we have
$$
\delta(u\cdot v)=(\delta u)\cdot v+(-1)^{|u|}u\cdot (\delta v)
$$
and
$$
\delta(u\cdot c)=(\delta u)\cdot c+(-1)^{|u|}u\cdot (\delta c)
$$where $|u|$ is  the total degree of $u$.
\end{prop}
\begin{proof}
This is a routine check.
\end{proof}

Now we consider a ringed space $(X, \mathcal{R})$ and an open cover $\mathcal{U}$ of $X$. The classifying space $\mathcal{N}$ of $\mathcal{U}$ as in Example \ref{eg: classifying space of an open cover} is a simplicial ringed space with structure sheaves inherited from the $\mathcal{R}$ on $X$ and we denote this simplicial ringed space by $(\mathcal{N}, \mathcal{R})$. In this case we have the following observations. Actually they are exactly the conventions in \cite[Section 1]{o1981trace}.
\begin{itemize}
\item An element $c^{p,q}$ of $C^p(\mathcal{N},\mathcal{R}, E^q)$ consists of a section $c^{p,q}_{i_0\ldots i_p}$ of $E^{q}_{i_0}$ over each non-empty intersection $U_{i_0\ldots i_p}$. If $U_{i_0\ldots i_p}=\emptyset$, then the component on it is zero. 
\item An element $u^{p,q}$ of $C^p(\mathcal{N},\mathcal{R},\text{Hom}^q(E,F))$ gives a section $u^{p,q}_{i_0\ldots i_p}$ of $\text{Hom}^q_{\mathcal{R}_p-\text{Mod}}(E^{\bullet}_{i_p},F^{\bullet}_{i_0})$, i.e. a degree $q$ map from $E^{\bullet}_{i_p}$ to $F^{\bullet}_{i_0}$ over the non-empty intersection $U_{i_0\ldots i_p}$. Notice that we require $u^{p,q}$ to be a map from the $F^{\bullet}$ on the last subscript of $U_{i_0\ldots i_p}$ to the $E^{\bullet}$ on the first subscript of $U_{i_0\ldots i_p}$. Again, if $U_{i_0\ldots i_p}=\emptyset$,  then the component on it is zero.
\end{itemize}
The compositions and actions are given in the following formula  (see \cite[Equation (1.1) and Equation (1.2)]{o1981trace}):
$$
(u\cdot v)^{p+r,q+s}_{i_0\ldots i_{p+r}}=(-1)^{qr}u^{p,q}_{i_0\ldots i_p}v^{r,s}_{i_p\ldots i_{p+r}}
$$
and
$$
(u\cdot c)^{p+r,q+s}_{i_0\ldots i_{p+r}}=(-1)^{qr}u^{p,q}_{i_0\ldots i_p}c^{r,s}_{i_p\ldots i_{p+r}}.
$$
Moreover the differentials are given by:
$$
(\delta u)^{p+1,q}_{i_0\ldots i_{p+1}}=\sum_{k=1}^p(-1)^k u^{p,q}_{i_0\ldots \widehat{i_k} \ldots i_{p+1}}|_{U_{i_0\ldots i_{p+1}}} \,\text{ for } u^{p,q}\in C^p(\mathcal{N},\mathcal{R},\text{Hom}^q(E,F))
$$
and
$$
(\delta c)^{p+1,q}_{i_0\ldots i_{p+1}}=\sum_{k=1}^{p+1}(-1)^k c^{p,q}_{i_0\ldots \widehat{i_k} \ldots i_{p+1}}|_{U_{i_0\ldots i_{p+1}}} \,\text{ for }c^{p,q}\in C^p(\mathcal{N},\mathcal{R}, E).
$$

\subsection{Twisted complexes}\label{subsection: twisted complexes}
With the notations in Section \ref{subsection: notation of bicomplexes} we can define twisted complexes on simplicial ringed spaces.
 \begin{defi}\label{defi: twisted complexes on simplicial spaces}
Let $(\mathcal{U},\mathcal{R})$ be a simplicial   ringed space. A twisted complex on  $(\mathcal{U},\mathcal{R})$ consists of a graded sheaf of $\mathcal{R}_0$-modules $E^{\bullet}$ on $U_0$ together with
$$
a=\prod_{k\geq 0} a^{k,1-k}\in \prod_{k\geq 0} \Hom^{1-k}_{\mathcal{R}_k-\text{Mod}}(\tau_{k,0}^*(E),\rho_{k,0}^*(E))
$$
where
$$
a^{k,1-k}\in \Hom^{1-k}_{\mathcal{R}_k-\text{Mod}}(\tau_{k,0}^*(E),\rho_{k,0}^*(E))
$$
and they satisfy the following two conditions.
\begin{enumerate}
\item The Maurer-Cartan equation
\begin{equation}\label{eq: MC for simplicial space}
\delta a+a\cdot a=0,
\end{equation}
or more explicitly
\begin{equation}\label{eq: MC for simplicial space explicit}
\sum_{j=1}^{k-1}(-1)^j\partial_j^*(a^{k-1,2-k})+\sum_{j=0}^k(-1)^{(1-j)(k-j)}\rho_{k,j}^*(a^{j,1-j})\tau_{k,k-j}^*(a^{k-j,1-k+j})=0;
\end{equation}

\item The non-degenerate condition: $a^{1,0}\in \Hom^0_{\mathcal{R}_1-\text{Mod}}(\tau_{1,0}^*(E),\rho_{1,0}^*(E))$ is invertible up to homotopy.
\end{enumerate}

A morphism $\theta$ of degree $m$ from $(E,a)$ to $(F,b)$ is given by a collection
$$
\theta^{k,m-k}\in \Hom^{m-k}_{\mathcal{R}_k-\text{Mod}}(\tau_{k,0}^*(E),\rho_{k,0}^*(F)) \text{ for all } k\geq 0
$$
and the differential is given by
$$
d\theta=\delta\theta+b\cdot \theta-(-1)^m\theta\cdot a
$$
or more explicitly
\begin{equation}
\begin{split}
(d\theta)^{k,m+1-k}&=\sum_{j=1}^{k-1}(-1)^j\partial^*_j\theta^{k-1,m+1-k}\\
+\sum_{l=0}^k&(-1)^{(1-l)(k-l)}\rho_{k,l}^*b^{l,1-l}\tau_{k,k-l}^*\theta^{k-l,m-k+1}+\sum_{l=0}^{k}(-1)^{(m-l)(k-l)}\rho_{k,l}^*\theta^{l,m-l}\tau_{k,k-l}^*a^{k-l,1-k+l}.
\end{split}
\end{equation}

We denote the dg-category of twisted complexes on a simplicial ringed space $(\mathcal{U},\mathcal{R})$ by Tw$(\mathcal{U}, \mathcal{R})$.
 \end{defi}

\begin{rmk}
People who are familiar with $A_{\infty}$-categories may find that the definition of twisted complexes is similar to the construction of $A_{\infty}$-functors. Actually this is the approach taken by \cite{tsygan2013microlocal} and \cite{arkhipov2018homotopy2}. In this paper we satisfy ourselves with Definition \ref{defi: twisted complexes on simplicial spaces} and refer interested readers to \cite[Section 16]{tsygan2013microlocal} and \cite[Section 4]{arkhipov2018homotopy2} for the $A_{\infty}$-approach.
\end{rmk}

\begin{defi}\label{defi: pull back of twisted complexes}
Let $f: (\mathcal{U},\mathcal{R})\to (\mathcal{V}, \mathcal{S})$ be a simlicial map between simplicial ringed spaces. Then $f$ naturally induces a dg-functor $f^*:\Tw(\mathcal{V},\mathcal{S})\to \Tw(\mathcal{U},\mathcal{R})$. More precisely, for $\mathcal{E}=(E,a)\in \Tw(\mathcal{V},\mathcal{S})$, $f^*\mathcal{E}$ is given by $(f_0^*E,f^*a)$ where
$$
(f^*a)^{k,1-k}=f_k^*a^{k,1-k}\in \Hom^{1-k}_{\mathcal{R}_k}(\tau_{k,0}^*(f_0^*E),\rho_{k,0}^*(f_0^*E)).
$$
For a degree $m$ morphism $\phi: \mathcal{E}\to \mathcal{F}$ we define $f^*\phi: f^*\mathcal{E}\to f^*\mathcal{F}$ as
$$
(f^*\phi)^{k,m-k}=f_k^*\phi^{k,m-k}\in \Hom^{m-k}_{\mathcal{R}_k}(\tau_{k,0}^*(f_0^*E),\rho_{k,0}^*(f_0^*F)).
$$
By Lemma \ref{lemma: rho tau f} this definition makes sense. It is clear that $\delta(f^*\phi)=f^*(\delta \phi)$ and $f^*(\phi\cdot\psi)=f^*\phi \cdot f^*\psi$.
\end{defi}

In the case that the simplicial space is the classifying space $\mathcal{N}$ of an open cover $\mathcal{U}$ as in Example \ref{eg: classifying space of an open cover} we have the following more concrete description of twisted complexes.

\begin{defi}[\cite{o1981trace} Definition 1.3 or \cite{wei2016twisted} Definition 5]\label{defi: twisted complex}
Let $(X,\mathcal{R})$ be a ringed  space and $\mathcal{U}=\{U_i\}$ be a locally finite open cover of $X$. A \emph{twisted complex} consists of  graded sheaves $E^{\bullet}_{i}$ of  $\mathcal{R}$-modules on each $U_i$ together with
a collection of morphisms  for $k\geq 0$ and every multi-index $(i_0\ldots i_k)$
 $$
a^{k,1-k}_{i_0\ldots i_k}\in \text{Hom}^{1-k}_{U_{i_0\ldots i_k}}(E_{i_k},E_{i_0}) 
$$
which satisfy the Maurer-Cartan equation
\begin{equation}\label{equation: MC for twisted complex explicit}
\sum_{j=1}^{k-1}(-1)^ja^{k-1,2-k}_{i_0\ldots\widehat{ i_j}\ldots i_k}+ \sum_{l=0}^k (-1)^{(1-l)(k-l)}a^{l,1-l}_{i_0\ldots i_l} a^{k-l,1-k+l}_{i_l\ldots i_k}=0.
\end{equation}

Moreover we impose the following non-degenerate condition: for each $i$, the chain map
\begin{equation}\label{equation: nondegenerate condition of aii}
a^{1,0}_{ii}: (E^{\bullet}_i,  a^{0,1}_i)\to (E^{\bullet}_i,  a^{0,1}_i) \text { is invertible up to homotopy.}
\end{equation}

Morphisms and differentials are defined similarly.
\end{defi}

For more details on twisted complexes see \cite{wei2016twisted}. In this paper we just mention the relation between twisted complexes and homotopy limits. Let
$$\Cpx: \text{Ringed Space}^{op}\rightarrow \dgCat$$
be the contravariant functor which assigns to each  ringed space $(X,\mathcal{R})$ the dg-category of complexes of left $\mathcal{R}$-modules on $X$. This is a presheaf of dg-categories. For a simplicial ringed space   $(\mathcal{U},\mathcal{R})$ we get a cosimplicial diagram of dg-categories
\begin{equation}\label{equation: second cosimplicial diagram of Cpx of an open cover}
\begin{tikzcd}
\Cpx(U_0, \mathcal{R}_0) \arrow[yshift=0.7ex]{r}\arrow[yshift=-0.7ex]{r}& \Cpx(U_1, \mathcal{R}_1) \arrow[yshift=1ex]{r}\arrow{r}\arrow[yshift=-1ex]{r}  &   \Cpx(U_2, \mathcal{R}_2) \cdots
\end{tikzcd}
\end{equation}
Then we have the following result.
\begin{prop}\label{prop: twisted complex is homotopy limit}[\cite[Corollary 4.8]{block2017explicit}, \cite[Proposition 4.0.2]{arkhipov2018homotopy2}]
Let $\mathcal{U}$ be a  simplicial ringed  space . Then
the dg-category of twisted complexes Tw$(\mathcal{U}, \mathcal{R})$ gives an explicit construction of  $\holim \Cpx(\mathcal{U},\mathcal{R})$.
\end{prop}
Proposition \ref{prop: twisted complex is homotopy limit} shows the importance of twisted complexes in descent theory, See \cite[Introduction]{arkhipov2018homotopy2} for some discussions and \cite{wei2018descent} for an application.

\begin{rmk}
In practice we are often less interested in the category of all complexes  of $\mathcal{R}$-modules than in some well-behaved subcategory, say complexes with quasi-coherent cohomology on a scheme, or $\mathcal D_{X}$-modules which are quasi-coherent as $\mathcal O_{X}$-modules. As long as the condition we impose is local the theory works equally well in those cases. We will explicitly consider the case of perfect complexes in Section \ref{section: simplicial homotopy and twisted perfect complexes}.
\end{rmk}

 For later purpose we need the following concept. See \cite{wei2016twisted} Definition 2.27.

\begin{defi}\label{defi: weak equivalence between twisted complexes}
Let   $(\mathcal{U},\mathcal{R})$ be a simplicial ringed space. Let $\mathcal{E}=(E^{\bullet},a)$ and $\mathcal{F}=(F^{\bullet},b)$ be two objects in Tw$(\mathcal{U}, \mathcal{R})$. A morphism $\phi: \mathcal{E}\to \mathcal{F}$ is called a \emph{weak equivalence} if it satisfies the following two
conditions:
\begin{itemize}
\item $\phi$ is closed and of degree zero.
\item Its $(0,0)$ component $\phi^{0,0}: (E^{\bullet},a^{0,1})\to (F^{\bullet},b^{0,1})$ is a quasi-isomorphism of complexes of $\mathcal{R}_0$-modules on $U_0$.
\end{itemize}
\end{defi}

\section{$A_{\infty}$-natural transformations}\label{section: A infinity nt}
In this section we review  $A_{\infty}$-natural transformations between dg-functors. For more details see \cite{wei2019recurrent}. See for example \cite{lyubashenko2003category} or \cite{arkhipov2018homotopy2} for an introduction of more general $A_{\infty}$-categories, $A_{\infty}$-functors and $A_{\infty}$-natural transformations. Since we restrict ourselves to $A_{\infty}$-natural transformations between dg-functors, the notations and $\pm$ sign conventions of $A_{\infty}$-natural transformations can be dramatically simplified.

\begin{defi}[$A_{\infty}$-prenatural transformation]\label{def: A_infty-prenatural transformation}
Let  $F$, $G: \mathcal{C}\to \mathcal{D}$ be two dg -functors between dg-categories. An \emph{$A_{\infty}$-prenatural transformation} $\Phi: F\Rightarrow G$ of degree $n$ consists of the following data:
\begin{enumerate}
\item For any object $X\in \obj(\mathcal{C})$, a morphism $\Phi^0_X\in  \mathcal{D}^n(FX,GX)$;
\item For any $l\geq 1$ and any objects $X_0,\ldots, X_l\in \obj(\mathcal{C})$, a morphism 
$$
\Phi^l_{X_0,\ldots, X_l}\in \Hom^{n-l}_k(\mathcal{C}(X_{l-1},X_l)\otimes \ldots \otimes \mathcal{C}(X_0,X_1), \mathcal{D}(FX_0,GX_l))
$$
\end{enumerate}
\end{defi}

\begin{defi}[Differential of $A_{\infty}$-prenatural transformation]\label{def: differential of A_infty-prenatural transformation}
Let   $F$, $G: \mathcal{C}\to \mathcal{D}$ be two dg-functors between dg-categories. Let $\Phi: F\Rightarrow G$ be an $A_{\infty}$-prenatural transformation of degree $n$ as in Definition \ref{def: A_infty-prenatural transformation}. Then the \emph{differential} $d\Phi: F\Rightarrow G$ is an  $A_{\infty}$-prenatural transformation of degree $n+1$ whose components are given as follows:
\begin{enumerate}
\item For any object $X\in \obj(\mathcal{C})$,$(d^{\infty}\Phi)^0_X=d(\Phi^0_X) \in  \mathcal{D}^{n+1}(FX,GX)$;
\item For any $l\geq 1$ and a collection of morphisms $u_i\in  \mathcal{C}(X_{i-1},X_i)$ $i=1,\ldots, l$,
\begin{equation}\label{eq: differential of A infty prenatural transformation}
\begin{split}
&(d^{\infty}\Phi)^l(u_l\otimes\ldots \otimes u_1)=\\
&d(\Phi^l(u_l\otimes\ldots \otimes u_1))+(-1)^{|u_l|-1}G(u_l)\Phi^{l-1}(u_{l-1}\otimes \ldots \otimes u_1)\\
&+(-1)^{n|u_1|-|u_1|-\ldots-|u_l|+l-1}\Phi^{l-1}(u_l\otimes\ldots\otimes u_2)F(u_1)\\
&+\sum_{i=1}^l(-1)^{|u_l|+\ldots+|u_{i+1}|+l-i+1}\Phi^{l}(u_l\otimes\ldots\otimes du_i\otimes \ldots \otimes u_1)\\
&+\sum_{i=1}^{l-1}(-1)^{|u_l|+\ldots+|u_{i+1}|+l-i+1}\Phi^{l-1}(u_l\otimes\ldots\otimes u_{i+1}u_i\otimes \ldots  \otimes u_1)
\end{split}
\end{equation}
\end{enumerate}
\end{defi}

\begin{rmk}
The last term in \eqref{eq: differential of A infty prenatural transformation} exists only if $l\geq 2$.
\end{rmk}

\begin{rmk}
The $d^{\infty}$ above is differed from the $\mu^1_{\mathcal{Q}}$ in \cite[Section I.1d]{seidel2008fukaya} by $(-1)^{n-l+|u_1|+\ldots+|u_l|}$ on each term, which does not infect the properties of $d^{\infty}$.
\end{rmk}

We can check that $d^{\infty}\circ d^{\infty}=0$ on $A_{\infty}$-prenatural transformations.

\begin{defi}[$A_{\infty}$-natural transformation]\label{def: A_infty natural transformation}
Let  $F$, $G: \mathcal{C}\to \mathcal{D}$ be two dg-functors between dg-categories. Let $\Phi: F\Rightarrow G$ be an $A_{\infty}$-prenatural transformation. We call $\Phi$ an \emph{$A_{\infty}$-natural transformation} if $\Phi$ is of degree $0$ and closed under the differential $d^{\infty}$ in Definition \ref{def: differential of A_infty-prenatural transformation}.
\end{defi}

For an $A_{\infty}$-natural transformation $\Phi: F\Rightarrow G$, the $l=0$ component of \eqref{eq: differential of A infty prenatural transformation} is simply $d(\Phi^0_X)=0$ for any object $X$. The $l=1$ condition is that for any $u\in \mathcal{C}(X_0,X_1)$ we have
\begin{equation}\label{eq: l=1 for A-infty natural transformation}
d(\Phi^1(u))-\Phi^1(d(u))+(-1)^{|u|}\Phi^0_{X_1}F(u)+(-1)^{|u|+1}G(u)\Phi^0_{X_0}=0.
\end{equation}
The $l=2$ condition is that for any $u_1\in \mathcal{C}(X_0,X_1)$  and  $u_2\in \mathcal{C}(X_1,X_2)$ we have
\begin{equation}\label{eq: l=2 for A-infty natural transformation}
\begin{split}
d(\Phi^2&(u_2\otimes u_1))-(-1)^{|u_1|+|u_2|}\Phi^1(u_2)F(u_1)-(-1)^{|u_2|}G(u_2)\Phi^1(u_1)\\
&+(-1)^{|u_2|}\Phi^2(u_2\otimes du_1)-\Phi^2(du_2\otimes u_1)+(-1)^{|u_2|}\Phi^1(u_2u_1)=0.
\end{split}
\end{equation}

It is clear that a closed degree $0$ dg-natural transformation $\Phi$ can be considered as an $A_{\infty}$-natural transformation  with $\Phi^l=0$ for all $l\geq 1$.

\begin{defi}[Compositions]\label{def: composition of A-infty natural transformation}
Let $F$, $G$, $H: \mathcal{C}\to \mathcal{D}$ be three dg-functors between dg-categories. Let $\Phi: F\Rightarrow G$  and $\Psi: G\Rightarrow H$ be two $A_{\infty}$-natural transformations. Then the composition $\Psi\circ \Phi$ is defined as follows: For any object $X\in \obj(\mathcal{C})$
$$
(\Psi\circ \Phi)^0_X:=\Psi^0_X\Phi^0_X: FX\to GX\to HX
$$
and for any $u_i\in \mathcal{C}(X_{i-1},X_i)$, $i=1,\ldots, l$
\begin{equation*}
\begin{split}
(\Psi\circ \Phi)^l(u_l\otimes\ldots \otimes u_1):=&\sum_{k=1}^{l-1}\Psi^{l-k}(u_l\otimes\ldots \otimes u_{k+1}) \Phi^k(u_k\otimes\ldots  \otimes u_1)\\
+&\Psi^l(u_l\otimes\ldots \otimes u_1) \Phi^0_{X_0}+\Psi^0_{X_l} \Phi^l(u_l\otimes\ldots  \otimes u_1).
\end{split}
\end{equation*}
We can check that $\Psi\circ \Phi$ is an $A_{\infty}$-natural transformation.
\end{defi}

\begin{rmk}
We can define compositions for general $A_{\infty}$-prenatural transformations. See \cite[Section 3]{lyubashenko2003category} or \cite[Section I.1(d)]{seidel2008fukaya}.
\end{rmk}

\begin{defi}[$A_{\infty}$-quasi-inverse]\label{def: A-infty quasi-inverse}
Let  $F$, $G: \mathcal{C}\to \mathcal{D}$ be two dg $k$-functors between dg-categories. Let $\Phi: F\Rightarrow G$ be an $A_{\infty}$-natural transformation. We call an  $A_{\infty}$-natural transformation $\Psi: G\Rightarrow F$ an \emph{$A_{\infty}$-quasi-inverse} of $\Phi$ if there exist $A_{\infty}$-prenatural transformations $\eta: F\Rightarrow F$ and $\omega: G\Rightarrow G$ both of degree $-1$ such that
$$
\Psi\circ \Phi-\id_F=d^{\infty}\eta, \text{ and } \Phi\circ \Psi-\id_G=d^{\infty}\omega.
$$
In more details, this means that we have
$$
\Psi^0_X\Phi^0_X-\id_{FX}=d\eta^0_X, \text{ and } \Phi^0_X\Psi^0_X-\id_{GX}=d\omega^0_X \text{ for any } X\in \obj{\mathcal{C}}
$$
and for any $l\geq 1$ and any $u_i\in \mathcal{C}(X_{i-1},X_i)$, $i=1,\ldots, l$, we have
\begin{equation}
\begin{split}
&\sum_{k=1}^{l-1}\Psi^{l-k}(u_l\otimes\ldots \otimes u_{k+1}) \Phi^k(u_k\otimes\ldots  \otimes u_1)\\
+&\Psi^l(u_l\otimes\ldots \otimes u_1) \Phi^0_{X_0}+\Psi^0_{X_l} \Phi^l(u_l\otimes\ldots  \otimes u_1)=\\
&d(\eta^l(u_l\otimes\ldots \otimes u_1))+(-1)^{|u_l|-1}G(u_l)\eta^{l-1}(u_{l-1}\otimes \ldots \otimes u_1)\\
&+(-1)^{-|u_2|-\ldots-|u_l|+l-1}\eta^{l-1}(u_l\otimes\ldots\otimes u_2)F(u_1)\\
&+\sum_{i=1}^l(-1)^{|u_l|+\ldots+|u_{i+1}|+l-i+1}\eta^{l}(u_l\otimes\ldots du_i\otimes \ldots u_1)\\
&+\sum_{i=1}^{l-1}(-1)^{|u_l|+\ldots+|u_{i+1}|+l-i+1}\eta^{l-1}(u_l\otimes\ldots u_{i+1}u_i\otimes \ldots u_1)
\end{split}
\end{equation}
and
\begin{equation}
\begin{split}
&\sum_{k=1}^{l-1}\Phi^{l-k}(u_l\otimes\ldots \otimes u_{k+1}) \Psi^k(u_k\otimes\ldots  \otimes u_1)\\
+&\Phi^l(u_l\otimes\ldots \otimes u_1) \Psi^0_{X_0}+\Phi^0_{X_l} \Phi^l(u_l\otimes\ldots  \otimes u_1)\\
&d(\omega^l(u_l\otimes\ldots \otimes u_1))+(-1)^{|u_l|-1}G(u_l)\omega^{l-1}(u_{l-1}\otimes \ldots \otimes u_1)\\
&+(-1)^{-|u_2|-\ldots-|u_l|+l-1}\omega^{l-1}(u_l\otimes\ldots\otimes u_2)F(u_1)\\
&+\sum_{i=1}^l(-1)^{|u_l|+\ldots+|u_{i+1}|+l-i+1}\omega^{l}(u_l\otimes\ldots du_i\otimes \ldots u_1)\\
&+\sum_{i=1}^{l-1}(-1)^{|u_l|+\ldots+|u_{i+1}|+l-i+1}\omega^{l-1}(u_l\otimes\ldots u_{i+1}u_i\otimes \ldots u_1)
\end{split}
\end{equation}
\end{defi}

\begin{prop}\label{prop: A infinity quasi-inverse and objectwise homotopy inverse}
Let  $F$, $G: \mathcal{C}\to \mathcal{D}$ be two dg-functors between dg-categories and $\Phi: F\Rightarrow G$ be an $A_{\infty}$-natural transformation. Then $\Phi$ admits an $A_{\infty}$-quasi-inverse if and only if  $\Phi^0_X:FX\to GX$ is invertible in the homotopy category Ho$\mathcal{D}$ for any object $X\in \mathcal{C}$.
\end{prop}
\begin{proof}
See \cite[Proposition 7.15]{lyubashenko2003category} or \cite[Theorem 4.1]{wei2019recurrent}.
\end{proof}

\begin{rmk}
Proposition 7.15 in \cite{lyubashenko2003category} is a more general result on  $A_{\infty}$-natural transformation of  $A_{\infty}$-functors between  $A_{\infty}$-categories.
\end{rmk}

\section{Simplicial homotopies}\label{section: Simplicial homotopies and twisted complexes}

\subsection{A review of simplicial homotopies}\label{subsection: review of simplicial homotopy}
First we review the definition of simplicial homotopies between simplicial maps. For more details see  \cite{goerss2009simplicial} Section I.6.
\begin{defi}\label{defi: simplicial homotopy between simplicial spaces}
Let $\mathcal{C}$ be a category which admits finite colimits.  For a simplicial object $\mathcal{U}$ in $\mathcal{C}$, we can construct the tensor product
$\mathcal{U}\times I$ where $I$ is the simplicial set $\Delta_1$. Two simplicial maps $f,g: \mathcal{U}\to \mathcal{V}$ between simplicial objects are called \emph{simplicial homotopic} if there is a map $H: \mathcal{U}\times I\to \mathcal{V}$ such that
$$
f=H\circ\varepsilon_0 \text{ and } g=H\circ\varepsilon_1
$$
where $\varepsilon_{\mu}: \mathcal{U}\to \mathcal{U}\times I$, $\mu=0,1$ are the two obvious inclusions. In this case we call $H$ a simplicial homotopy between $f$ and $g$.
\end{defi}

\begin{rmk}
In the literature a simplicial homotopy is sometimes called a strict simplicial homotopy. In Definition \ref{defi: simplicial homotopy between simplicial spaces} we simply call it simplicial homotopy. Nevertheless we notice that simplicial homotopy is not an equivalence relation if we put no restriction on $\mathcal{V}$. See \cite[Section I.6]{goerss2009simplicial} for further discussions.
\end{rmk}

We have the following equivalent definition of simplicial homotopy, which is useful in the proof of Proposition \ref{prop: homotopic functors induces quasi-isomorphic dg-functors between twisted complexes} below. See \cite{may1992simplicial} Definition 5.1.

\begin{defi}\label{defi: combinatorial simplicial homotopy between simplicial spaces}
Two maps $f,g:  \mathcal{U}\to \mathcal{V}$ between simplicial objects are called \emph{combinatorial simplicial homotopic} if for each $p\geq 0$, there exist morphisms
$$
h_i=h^p_i:U_p\to V_{p+1} \text{ for } i=0,\ldots,p
$$
such that the following conditions hold.
\begin{enumerate}
\item $$\partial_0 h_0=f_p, \partial_{p+1}h_p=g_p;$$
\item
$$
\partial_ih_j=\begin{cases}h_{j-1}\partial_i & i<j\\
\partial_ih_{i-1} &i=j\neq 0\\
h_j\partial_{i-1} & i>j+1
\end{cases};$$
\item
$$
s_ih_j=\begin{cases}h_{j+1}s_i & i\leq j\\
h_js_{i-1} & i>j
\end{cases}.
$$
\end{enumerate}
\end{defi}

\begin{lemma}\label{lemma: two versions of simplicial homotopy are equivalent}
Let $\mathcal{C}$ be a category which admits finite colimits. Then the two versions of simplicial homotopy in Definition \ref{defi: simplicial homotopy between simplicial spaces} and Definition \ref{defi: combinatorial simplicial homotopy between simplicial spaces} are equivalent.
\end{lemma}
\begin{proof}
It is an easy but complicated combinatorial check. See \cite{may1992simplicial} Proposition 6.2. The proof there is for $\mathcal{C}=Sets$ but it also works for general $\mathcal{C}$.
\end{proof}

\begin{lemma}\label{lemma: f, g and h}
For any $k\geq p$ let $\rho_{k,p}$ and $\tau_{k,p}$ be the front and back face maps as in \eqref{eq: rho} and \eqref{eq: tau}. We have the following identities.
\begin{equation}\label{eq: h and partial}
h_i\partial_j=\begin{cases}
    \partial_{j+1}h_i       & i<j\\
  \partial_j h_{i+1}  & i\geq j;
  \end{cases}
\end{equation}
\begin{equation}\label{eq: h tau}
h_i\circ \tau_{k,p}=\tau_{k+1,p+1}\circ h_{i+k-p},~\forall 0\leq i\leq p;
\end{equation}
\begin{equation}\label{eq: h rho}
h_i\circ \rho_{k,p}=\rho_{k+1,p+1}\circ h_i,~\forall 0\leq i\leq p;
\end{equation}
and
\begin{equation}\label{eq: f tau}
f\circ \tau_{k,p}=\tau_{k+1,p}\circ h_i, ~\forall 0\leq i\leq k-p;
\end{equation}
\begin{equation}\label{eq: g rho}
g\circ \rho_{k,p}=\rho_{k+1,p}\circ h_i, ~\forall p\leq i\leq k.
\end{equation}
\end{lemma}
\begin{proof}
It is a routine check of Definition \ref{defi: combinatorial simplicial homotopy between simplicial spaces} and the simplicial identities.
\end{proof}

\subsection{Simplicial homotopic maps and twisted complexes}
In the sequel we  compare $f^*$ and $g^*$ for simplicial homotopic maps $f$ and $g$.

\begin{prop}\label{prop: homotopic functors induces quasi-isomorphic dg-functors between twisted complexes}
Let $f$ and $g$ be two simplicial maps between simplicial ringed spaces $(\mathcal{U}, \mathcal{R})$ and  $(\mathcal{V},\mathcal{S})$. Let $h$ be a simplicial homotopy between $f$ and $g$ as in Definition \ref{defi: combinatorial simplicial homotopy between simplicial spaces}. Then for any twisted complex  $\mathcal{E}=(E^{\bullet},a)$ on $(\mathcal{V},\mathcal{S})$, the homotopy $h$ induces a weak equivalence
$$
\Phi_{0}(\mathcal{E}):f^*(\mathcal{E})\overset{\sim}{\to} g^*(\mathcal{E}).
$$
\end{prop}
\begin{proof}
For any $k\geq 0$ we have $a^{k+1,-k}\in \Hom^{-k}_{\mathcal{S}_{k+1}}(\tau_{k+1,0}^*(E),\rho_{k+1,0}^*(E))$. Using $h_i: (U_k,\mathcal{R}_k)\to (V_{k+1},\mathcal{S}_{k+1})$ we obtain
$$
h_i^*(a^{k+1,-k})\in \Hom^{-k}_{\mathcal{R}_k}(h_i^*\tau_{k+1,0}^*(E),h_i^*\rho_{k+1,0}^*(E))\text{ for } 0\leq i\leq k.
$$
By \eqref{eq: f tau} and \eqref{eq: g rho} we have
$$
\tau_{k+1,0}\circ h_i=f_0\circ\tau_{k,0} \text{ and } \rho_{k+1,0}\circ h_i=g_0\circ\rho_{k,0}
$$
hence we get
$$
h_i^*(a^{k+1,-k})\in \Hom^{-k}_{\mathcal{R}_k}(\tau_{k,0}^*f_0^*(E),\rho_{k,0}^*g_0^*(E)).
$$

Then we define $\Phi_{0}(\mathcal{E})$ as follows
\begin{equation}
\Phi^{k,-k}_{0}(\mathcal{E})=\sum_{i=0}^k(-1)^ih_i^*(a^{k+1,-k})\in \Hom^{-k}_{\mathcal{R}_k}(\tau_{k,0}^*f_0^*(E),\rho_{k,0}^*g_0^*(E)).
\end{equation}

\begin{lemma}\label{lemma: combinatorics of h}
For any $k\geq 0$ we have
\begin{equation}\label{eq: sum partial h}
\sum_{i=1}^{k-1}\sum_{j=0}^{k-1}(-1)^{i+j}\partial_i^*h_j^*=\sum_{i=1}^k\sum_{j=0}^k(-1)^{i+j-1}h_j^*\partial_i^*.
\end{equation}
Moreover for two morphisms $\phi:\mathcal{F}\to \mathcal{G}$ and $\psi: \mathcal{E}\to \mathcal{F}$ with degree $m$ and $n$ respectively, we have
\begin{equation}\label{eq: sum rho g}
\begin{split}
&\sum_{i=0}^{k}\sum_{j=0}^{k-i}(-1)^{(m-i)(k-i)+j}(\rho_{k,i}^*g^*\phi^{i,m-i})(\tau_{k,k-i}^*h_j^*\psi^{k-i+1,n-1+i-k})\\
=&\sum_{j=0}^k(-1)^{j+m}h_j^*[\sum_{i=0}^j(-1)^{(m-i)(k-i+1)}(\rho_{k+1,i}^*\phi^{i,m-i})(\tau_{k+1,k-i+1}^*\psi^{k-i+1,n-1+i-k})];
\end{split}
\end{equation}
and
\begin{equation}\label{eq: sum tau f}
\begin{split}
&\sum_{i=0}^{k}\sum_{j=0}^i(-1)^{(m-i-1)(k-i)+j}(\rho_{k,i}^*h_j^*\phi^{i+1,m-i-1})(\tau_{k,k-i}^*f^*\psi^{k-i,n+i-k})\\
=&\sum_{j=0}^k(-1)^jh_j^*[\sum_{i=j+1}^{k+1}(-1)^{(m-i)(k-i+1)}(\rho_{k+1,i}^*\phi^{i,m-i})(\tau_{k+1,k-i+1}^*\psi^{k-i+1,n-1+i-k})];
\end{split}
\end{equation}
\end{lemma}
\begin{proof}[The proof of Lemma \ref{lemma: combinatorics of h}]
These identities follow from Lemma \ref{lemma: f, g and h} and re-indexing.
\end{proof}

Then we prove that the morphism $\Phi_{0}(\mathcal{E})$ is closed, i.e.  for any $k\geq 0$ we have
\begin{equation}\label{eq: Phi is closed}
\delta \Phi_{0}(\mathcal{E})+g^*(a)\cdot \Phi_{0}(\mathcal{E})-\Phi_{0}(\mathcal{E})\cdot f^*(a)=0.
\end{equation}
First we have
\begin{equation*}
\begin{split}
(\delta&\Phi_{0}(\mathcal{E}))^{k,1-k}=\sum_{i=1}^{k-1}(-1)^i\partial^*_i\Phi^{k-1,1-k}_{0,\mathcal{E}}\\
&=\sum_{i=1}^{k-1}(-1)^i\partial^*_i\sum_{j=0}^{k-1}(-1)^jh_j^*a^{k,1-k}\\
&=\sum_{j=0}^{k-1}\sum_{i=1}^{k-1}(-1)^{i+j}\partial^*_ih_j^*a^{k,1-k}
\end{split}
\end{equation*}
By \eqref{eq: sum partial h} we have
\begin{equation*}
\begin{split}
&(\delta\Phi_{0}(\mathcal{E}))^{k,1-k}\\
&=\sum_{j=0}^k\sum_{i=1}^k(-1)^{i+j-1}h_j^*\partial_i^*a^{k,1-k}\\
&=\sum_{j=0}^k(-1)^{j-1}h_j^*\sum_{i=1}^k(-1)^i \partial_i^*a^{k,1-k}.
\end{split}
\end{equation*}
Similarly by \eqref{eq: sum rho g} we have
$$
(g^*(a)\cdot \Phi_{0}(\mathcal{E}))^{k,1-k}=\sum_{j=0}^k(-1)^{j-1}h_j^*\sum_{i=0}^j(-1)^{(1-i)(k+1-i)}\rho_{k+1,i}^*a^{i,1-i}\tau_{k+1,k+1-i}^*a^{k+1-i,i-k},
$$
and by \eqref{eq: sum tau f} we have
$$
(\Phi_{0}(\mathcal{E})\cdot f^*(a))^{k,1-k}=\sum_{j=0}^k(-1)^{j-1}h_j^*\sum_{i=j+1}^{k+1}(-1)^{(1-i)(k+1-i)}\rho_{k+1,i}^*a^{i,1-i}\tau_{k+1,k+1-i}^*a^{k+1-i,i-k}.
$$
Sum up these three identities and use $\delta a+a\cdot a=0$ we get \eqref{eq: Phi is closed}.

Finally we notice that $\Phi^{0,0}_{0}(\mathcal{E})=h_0^*(a^{1,0})\in \Hom^0_{\mathcal{R}_0}(\tau_{0,0}^*f_0^*(E),\rho_{0,0}^*g_0^*(E))$ is a quasi-isomorphism since $a^{1,0}\in \Hom^0_{\mathcal{S}_1}(\tau_{1,0}^*(E),\rho_{1,0}^*(E))$ is invertible up to homotopy. Therefore we know that $\Phi_{0}(\mathcal{E})$ is a weak equivalence.
\end{proof}

\begin{rmk}
In general  for a morphsim $\phi: \mathcal{E}\to \mathcal{F}$, 
$$g^*(\phi) \cdot\Phi_{0}(\mathcal{E})-(-1)^{|\phi|} \Phi_{0}(\mathcal{F})\cdot f^*(\phi)\neq 0.$$
Therefore $\Phi_{0}(-)$ does not give a dg-natural transformation from $f^*$ to $g^*$. Nevertheless we can extend $\Phi_{0}(-)$ to an $A_{\infty}$-natural transformation. 
\end{rmk}

\subsection{Simplicial homotopies and $A_{\infty}$-natural transformations}\label{subsection: simplicial homotopies and A infinity}
In this section we introduce higher $\Phi_l$'s. Consider a degree $m$ morphism $\phi: \mathcal{E}\to \mathcal{F}$ in $\Tw(\mathcal{V},\mathcal{S})$. For any $k\geq 0$ we have
$$
\phi^{k+1,m-k-1}\in \Hom^{m-k-1}_{\mathcal{S}_{k+1}}(\tau_{k+1,0}^*(E),\rho_{k+1,0}^*(F)).
$$
Hence for $0\leq i\leq k$ we have
$$
h_i^*\phi^{k+1,m-k-1}\in \Hom^{m-k-1}_{\mathcal{R}_k}(h_i^*\tau_{k+1,0}^*(E),h_i^*\rho_{k+1,0}^*(F))
$$
and by \eqref{eq: f tau} and \eqref{eq: g rho} we have
$$
h_i^*\phi^{k+1,m-k-1}\in \Hom^{m-k-1}_{\mathcal{R}_k}(\tau_{k,0}^*f^*(E),\rho_{k,0}^*g^*(F))
$$

Now we are ready for the following definition.

\begin{defi}\label{defi: higher terms of A infinity nt}
For a degree $m$ morphism  $\phi: \mathcal{E}\to \mathcal{F}$ in $\Tw(\mathcal{V},\mathcal{S})$, we define $\Phi_1(\phi): f^*\mathcal{E}\to g^*\mathcal{F}$ as
\begin{equation}
[\Phi_1(\phi)]^{k,m-k-1}:=(-1)^{m-1}\sum_{i=0}^k(-1)^ih_i^*\phi^{k+1,m-k-1}.
\end{equation}
For $l\geq 2$ we simply define $\Phi_l=0$.
\end{defi}

We need to prove that $\Phi_0$ and $\Phi_1$ together form an $A_{\infty}$-natural transformation from $f$ to $g$. First we prove the following proposition.

\begin{prop}\label{prop: Phi is A infinity nt level 1}
For a degree $m$ morphism  $\phi: \mathcal{E}\to \mathcal{F}$ in $\Tw(\mathcal{V},\mathcal{S})$, we have
\begin{equation}
d[\Phi_1(\phi)]-\Phi_1(d\phi)+(-1)^{m-1}g^*(\phi)\Phi_0(\mathcal{E})+(-1)^m\Phi_0(\mathcal{F})f^*(\phi)=0
\end{equation}
\end{prop}
\begin{proof}
First we have
\begin{equation*}
\begin{split}
&[d\Phi_1(\phi)]^{k,m-k}\\
=&[\delta \Phi_1(\phi)]^{k,m-k}+[g^*(b)\cdot \Phi_1(\phi)]^{k,m-k}-(-1)^{m-1}[\Phi_1(\phi)\cdot f^*(a)]^{k,m-k}.
\end{split}
\end{equation*}
By definition 
\begin{equation*}
\begin{split}
[\delta \Phi_1(\phi)]^{k,m-k}&=\sum_{i=1}^{k-1}(-1)^i\partial_i^*\sum_{j=0}^{k-1}(-1)^{j+m-1}h_j^*\phi^{k,m-k}\\
&=\sum_{i=1}^{k-1}\sum_{j=0}^{k-1}(-1)^{i+j+m-1}\partial_i^*h_j^*\phi^{k,m-k}
\end{split}
\end{equation*}
and by \eqref{eq: sum partial h} we have
\begin{equation}\label{eq: delta Phi_1 phi}
[\delta \Phi_1(\phi)]^{k,m-k}=\sum_{i=0}^k(-1)^{i+m}h_i^*\sum_{j=1}^k(-1)^j\partial_j^*\phi^{k,m-k}.
\end{equation}
Next
\begin{equation*}
\begin{split}
&[g^*(b)\cdot \Phi_1(\phi)]^{k,m-k}\\
=&\sum_{i=0}^k(-1)^{(1-i)(k-i)}[\rho_{k,i}^*g^*(b)^{i,1-i}][ \tau_{k,k-i}^*\Phi_1(\phi)^{k-i,m-1-k+i}]\\
=&\sum_{i=0}^k\sum_{j=0}^{k-i}(-1)^{(1-i)(k-i)+j+m-1}\rho_{k,i}^*g^*b^{i,1-i}\tau_{k,k-i}^*h_j^*\phi^{k-i+1,m-1-k+i}
\end{split}
\end{equation*}
and by \eqref{eq: sum rho g} we have
\begin{equation}
[g^*(b)\cdot \Phi_1(\phi)]^{k,m-k}=\sum_{i=0}^k(-1)^{i+m}h_i^*\sum_{j=0}^i(-1)^{(1-j)(k-j+1)}\rho_{k+1,j}^*b^{j,1-j}\tau_{k+1,k+1-j}^*\phi^{k-j+1,m-1-k+j}.
\end{equation}
Similarly by \eqref{eq: sum tau f} we have
\begin{equation}
[\Phi_1(\phi)\cdot f^*(a)]^{k,m-k}=\sum_{i=0}^k(-1)^{i+m-1}h_i^*\sum_{j=i+1}^{k+1}(-1)^{(m-j)(k-j+1)}\rho_{k+1,j}^*\phi^{j,m-j}\tau_{k+1,k+1-j}^*a^{k-j+1,j-k}.
\end{equation}

Again by \eqref{eq: sum rho g} we get
\begin{equation}
[g^*(\phi)\Phi_0(\mathcal{E})]^{k,m-k}=\sum_{i=0}^k(-1)^{i-m}h_i^*\sum_{j=0}^i(-1)^{(m-j)(k-j+1)}\rho_{k+1,j}^*\phi^{j,m-j}\tau_{k+1,k+1-j}^*a^{k-j+1,j-k}
\end{equation}
and by \eqref{eq: sum tau f} we get
\begin{equation}\label{eq: Phi_0 F f phi}
[\Phi_0(\mathcal{F})f^*(\phi)]^{k,m-k}=\sum_{i=0}^k(-1)^ih_i^*\sum_{j=i+1}^{k+1}(-1)^{(1-j)(k-j+1)}\rho_{k+1,j}^*b^{j,1-j}\tau_{k+1,k+1-j}^*\phi^{k-j+1,m-1-k+j}.
\end{equation}

Add up Equations \eqref{eq: delta Phi_1 phi} through \eqref{eq: Phi_0 F f phi} we get
\begin{equation}\label{eq: three terms in Phi_1 eq}
\begin{split}
&[d[\Phi_1(\phi)]]^{k,m-k}+(-1)^m[g^*(\phi)\Phi_0(\mathcal{E})]^{k,m-k}-(-1)^m[\Phi_0(\mathcal{F})f^*(\phi)]^{k,m-k}\\
=&(-1)^{m}\sum_{i=0}^k(-1)^ih_i^*\Bigg[\sum_{j=1}^k(-1)^j\partial_j^*\phi^{k,m-k}\\
+&\sum_{j=0}^{k+1}(-1)^{(1-j)(k-j+1)}\rho_{k+1,j}^*b^{j,1-j}\tau_{k+1,k+1-j}^*\phi^{k-j+1,m-1-k+j}\\
-&(-1)^m\sum_{j=0}^{k+1}(-1)^{(m-j)(k-j+1)}\rho_{k+1,j}^*\phi^{j,m-j}\tau_{k+1,k+1-j}^*a^{k-j+1,j-k}\Bigg].
\end{split}
\end{equation}
We observe that the right hand side of \eqref{eq: three terms in Phi_1 eq} is exactly $[\Phi_1(d\phi)]^{k,m-k}$, hence we complete the proof.
\end{proof}

\begin{prop}\label{prop: Phi is A infinity nt level 2}
For two morphisms $\phi:\mathcal{F}\to \mathcal{G}$ and $\psi: \mathcal{E}\to \mathcal{F}$ in $\Tw(\mathcal{V},\mathcal{S})$ with degree $m$ and $n$ respectively, we have
\begin{equation}
(-1)^{m-1}g^*(\phi)\cdot\Phi_1(\psi)+(-1)^{-m-n+1}\Phi_1(\phi)\cdot f^*(\psi)+(-1)^{m}\Phi_1(\phi\cdot \psi)=0
\end{equation}
\end{prop}
\begin{proof}
Again it is a consequence of \eqref{eq: sum rho g} and \eqref{eq: sum tau f} and the details are left to the readers.
\end{proof}

\begin{thm}\label{thm: simplicial homotopies and A infinity nt}
Let $f$ and $g$ be two simplicial maps between simplicial ringed spaces   $(\mathcal{U}, \mathcal{R})$ and  $(\mathcal{V},\mathcal{S})$. Let $h$ be a simplicial homotopy between $f$ and $g$ as in Definition \ref{defi: combinatorial simplicial homotopy between simplicial spaces}. Let $\Phi_0$ be as in Proposition \ref{prop: homotopic functors induces quasi-isomorphic dg-functors between twisted complexes} and $\Phi_1$ be as in Definition \ref{defi: higher terms of A infinity nt}. Then the collection $\Phi=\{\Phi_0,\Phi_1,0,0,\ldots\}$ is an $A_{\infty}$-natural transformation from $f^*$ to $g^*$.
\end{thm}
\begin{proof}
According to Definition \ref{def: differential of A_infty-prenatural transformation} and Definition  \ref{def: A_infty natural transformation}, all we need to prove is that $\Phi$ is closed under $d^{\infty}$, i.e. 
\begin{equation*}
\begin{split}
&d(\Phi_l(u_l\otimes\ldots \otimes u_1))+(-1)^{|u_l|-1}g^*(u_l)\cdot \Phi_{l-1}(u_{l-1}\otimes \ldots \otimes u_1)\\
&+(-1)^{-|u_1|-\ldots-|u_l|+l-1}\Phi_{l-1}(u_l\otimes\ldots\otimes u_2)\cdot f^*(u_1)\\
&+\sum_{i=1}^l(-1)^{|u_l|+\ldots+|u_{i+1}|+l-i+1}\Phi_{l}(u_l\otimes\ldots du_i\otimes \ldots u_1)\\
&+\sum_{i=1}^{l-1}(-1)^{|u_l|+\ldots+|u_{i+1}|+l-i+1}\Phi_{l-1}(u_l\otimes\ldots u_{i+1}\cdot u_i\otimes \ldots u_1)=0
\end{split}
\end{equation*}
for $l\geq 0$. According to \eqref{eq: l=1 for A-infty natural transformation} and \eqref{eq: l=2 for A-infty natural transformation},  for $l=0$, $1$, and $2$ these are consequences of Propositions \ref{prop: homotopic functors induces quasi-isomorphic dg-functors between twisted complexes}, \ref{prop: Phi is A infinity nt level 1}, and \ref{prop: Phi is A infinity nt level 2} respectively. For $l\geq 3$ it is trivial since $\Phi_l=0$ for $l\geq 2$.
\end{proof}

\begin{rmk}
It seems surprising why we can stop at $\Phi_1$. Actually in the definition of twisted complexes, we have only differential $a^{k,1-k}$ and maps $\phi^{k,m-k}$, whose pull back under $h$ give $\Phi^0$ and $\Phi^1$ respectively. Since the compositions of morphisms between twisted complexes are strictly associative, we can stop at $\Phi^1$. If the compositions were  weakly associative and we had higher associators, then we would have higher terms $\Phi^l$, $l\geq 2$ in the $A_{\infty}$-natural transformation. 
\end{rmk}

\section{Simplicial homotopies and twisted perfect complexes}\label{section: simplicial homotopy and twisted perfect complexes}
In this section we refine Theorem \ref{thm: simplicial homotopies and A infinity nt} for twisted perfect complexes. First we review the concept of twisted perfect complexes.

\subsection{A review of twisted perfect complexes}\label{subsection: twisted perfect complexes}
We are often not interested in all complexes of $\mathcal{R}$-modules but only some more convenient subcategory.
In this section we consider the contravariant functor
$$
\StrPerf: \text{Ringed Space}^{op}\rightarrow \dgCat
$$
which assigns to each ringed space $X$  the dg-category of strictly perfect complexes of $\mathcal{R}$-modules on $X$, i.e. bounded complexes of locally free finitely generated $\mathcal{R}$-modules on $X$. As before let $(\mathcal{U}, \mathcal{R})$ be a simplicial ringed space then we have a cosimplicial diagram of dg-categories.

\begin{equation}\label{equation: cosimplicial diagram of perfect Cpx of a simplicial space}
\begin{tikzcd}
\StrPerf(U_0,  \mathcal{R}_0) \arrow[yshift=0.7ex]{r}\arrow[yshift=-0.7ex]{r}&  \StrPerf(U_1,  \mathcal{R}_1) \arrow[yshift=1ex]{r}\arrow{r}\arrow[yshift=-1ex]{r}  &   \StrPerf(U_2,  \mathcal{R}_2) \cdots
\end{tikzcd}
\end{equation}

We have the following variant of twisted complexes.
\begin{defi}\label{defi: twisted perfect complex}
A \emph{twisted perfect complex} $\mathcal{E}=(E^{\bullet}_i,a)$ on a simplicial ringed space $(\mathcal{U}, \mathcal{R})$ is the same as twisted complex in Definition \ref{defi: twisted complexes on simplicial spaces} except that each $E^{\bullet}$ is a strictly perfect complex on $(U_0, \mathcal{R}_0)$.

The twisted perfect complexes also form a dg-category and we denote it by $\Tw_{\Perf}(\mathcal{U},\mathcal{R})$. Obviously $\Tw_{\Perf}(\mathcal{U}, \mathcal{R})$ is a full dg-subcategory of Tw$(\mathcal{U},\mathcal{R})$.
\end{defi}

\begin{lemma}\label{lemma: f star restrict to twisted perfect complexes}
Let $f: (\mathcal{U}, \mathcal{R})\to (\mathcal{V}, \mathcal{S})$ be a simplicial map between  simplicial ringed spaces. Then the dg-functor $f^*:\Tw(\mathcal{V},\mathcal{S})\to \Tw(\mathcal{U},\mathcal{R})$ restricts to the full dg-subcategory of twisted perfect complexes and gives a dg-functor
$$
f^*:\Tw_{\Perf}(\mathcal{V},\mathcal{S})\to \Tw_{\Perf}(\mathcal{U},\mathcal{R}).
$$
\end{lemma}
\begin{proof}
It is obvious since $f^*$ pulls back finitely generated locally free sheaves to finitely generated locally free sheaves.
\end{proof}

We have the following result for twisted perfect complexes which is similar to Proposition \ref{prop: twisted complex is homotopy limit}.
\begin{prop}\label{prop: twisted perfect complexes are totalization}
Let $\mathcal{U}$ be a  simplicial ringed  space . Then
the dg-category of twisted complexes $\Tw_{\Perf}(\mathcal{U}, \mathcal{R})$ gives an explicit construction of  $\holim \StrPerf(\mathcal{U})$.
\end{prop}

The significance of twisted perfect complexes in geometry is given by the construction in \cite{o1981hirzebruch}. Moreover, we have the following result:

\begin{thm}\label{thm: twisted complexes enhancement}[\cite[Theorem 3.32]{wei2016twisted}]
Let $X$ be a quasi-compact and separated
or Noetherian scheme and $\mathcal{U}=\{U_i\}$ be an affine cover, then $\Tw_{\Perf}(\mathcal{U},\mathcal{O}_X)$ gives a dg-enhancement of $D_{\Perf}(X)$, the derived category of perfect complexes on $X$.
\end{thm}
\begin{proof} See \cite{wei2016twisted} Theorem 3.32.
\end{proof}

The following proposition is about weak equivalences between twisted perfect complexes.

\begin{prop}\label{prop: weak equivalence between twisted perfect complexes}[\cite[Proposition 2.31]{wei2016twisted}]
Suppose the simplicial space $\mathcal{U}$ satisfies $H^k(U_i,\mathcal{S})=0$ for any $i\geq 0$, any  $k\geq 1$ and any locally free finitely generated sheaf of $\mathcal{R}_i$-modules $\mathcal{S}$. Let $\mathcal{E}$ and $\mathcal{F}$ be two objects in $\Tw_{\Perf}(\mathcal{U}, \mathcal{R})$ and $\phi: \mathcal{E}\to \mathcal{F}$ be a degree $0$ closed morphism. Then $\phi$ is a weak equivalence if and only if $\phi$ is invertible in the homotopy category Ho$\Tw_{\Perf}(\mathcal{U}, \mathcal{R})$.
\end{prop}
\begin{proof}
See \cite[Proposition 2.31]{wei2016twisted}.
\end{proof}

\begin{rmk}\label{rmk: examples of good covers}
The following simplicial spaces satisfy the condition in Proposition \ref{prop: weak equivalence between twisted perfect complexes}
\begin{itemize}
\item $X$ is a separated scheme and $\mathcal{U}=\{U_i\}$ is an affine cover of $X$;
\item $X$ is a complex manifold and $\mathcal{U}=\{U_i\}$ is a good cover of $X$ by Stein manifolds, i.e. all finite non-empty intersections of the cover are Stein manifolds. 
\end{itemize}
\end{rmk}

\begin{rmk}
In \cite[Proposition 2.31]{wei2016twisted} it requires that $H^k(U_i,\mathcal{S})=0$ for any $i\geq 0$, any  $k\geq 1$ and any \emph{quasi-coherent} sheaf of $\mathcal{R}_i$-modules $\mathcal{S}$ because it is based on  \cite[Lemma 2.30]{wei2016twisted} which requires the stronger condition. However, by a careful study we can see that the same proof of \cite[Proposition 2.31]{wei2016twisted} works if we only assume that  $H^k(U_i,\mathcal{S})=0$ for any $i\geq 0$, any  $k\geq 1$ and any locally free finitely generated sheaf of $\mathcal{R}_i$-modules $\mathcal{S}$. Nevertheless, the examples in Remark \ref{rmk: examples of good covers} satisfies the stronger condition too.
\end{rmk}

\begin{rmk}
The result in Proposition \ref{prop: weak equivalence between twisted perfect complexes} only applies to twisted perfect complexes because we need to use the fact that quasi-isomorphisms between bounded complexes of finitely generated projective modules have quasi-inverses, which fails for general complexes of modules.
\end{rmk}

\subsection{Simplicial homotopies and twisted perfect complexes}
Let $f$ and $g$ be two simplicial maps between simplicial ringed spaces $(\mathcal{U}, \mathcal{R})$ and  $(\mathcal{V},\mathcal{S})$. Let $h$ be a simplicial homotopy between $f$ and $g$ as in Definition \ref{defi: combinatorial simplicial homotopy between simplicial spaces}. By Lemma \ref{lemma: f star restrict to twisted perfect complexes} we have dg-functors $f^*$, $g^* :\Tw_{\Perf}(\mathcal{V},\mathcal{S})\to \Tw_{\Perf}(\mathcal{U}, \mathcal{R})$.

It is clear that the $A_{\infty}$-natural transformation $\Phi: f^*\Rightarrow g^*$ in Theorem \ref{thm: simplicial homotopies and A infinity nt} also restricts to twisted perfect complexes. Moreover we have the following result.

\begin{prop}
Let $f$ and $g$ be two simplicial maps between simplicial ringed spaces $(\mathcal{U}, \mathcal{R})$ and  $(\mathcal{V},\mathcal{S})$. Let $h$ be a simplicial homotopy between $f$ and $g$ as in Definition \ref{defi: combinatorial simplicial homotopy between simplicial spaces}.  Let $\Phi: f^*\Rightarrow g^*$ be the $A_{\infty}$-natural transformation as  in Theorem \ref{thm: simplicial homotopies and A infinity nt}. In addition assume $(\mathcal{U},\mathcal{R})$ satisfies $H^k(U_i,\mathcal{S})=0$ for any $i\geq 0$, any  $k\geq 1$ and any locally free finitely generated sheaf of $\mathcal{R}_i$-modules $\mathcal{S}$. Then $\Phi$ admits an $A_{\infty}$-quasi-inverse.
\end{prop}
\begin{proof}
By Proposition \ref{prop: homotopic functors induces quasi-isomorphic dg-functors between twisted complexes} $\Phi_0(\mathcal{E}): f^*(\mathcal{E})\to g^*(\mathcal{E})$ is a weak equivalence for each $\mathcal{E}$. By Proposition \ref{prop: weak equivalence between twisted perfect complexes} $\Phi_0(\mathcal{E})$ is invertible in the homotopy category for $\mathcal{E}\in \Tw_{\Perf}(\mathcal{V},\mathcal{S})$. The claim then follows from Proposition \ref{prop: A infinity quasi-inverse and objectwise homotopy inverse}.
\end{proof}

\begin{rmk}
Although the $A_{\infty}$-natural transformation $\Phi$ consists only two components $\Phi_0$ and $\Phi_1$, its  $A_{\infty}$-quasi-inverse may contain higher components.
\end{rmk}
\bibliography{homotopylimitbib}{}

\begin{thebibliography}{OTT81b}

\bibitem[A{\O}18]{arkhipov2018homotopy2}
Sergey Arkhipov and Sebastian {\O}rsted.
\newblock Homotopy limits in the category of dg-categories in terms of
  $\mathrm{A}_{\infty}$-comodules.
\newblock {\em arXiv preprint arXiv:1812.03583}, 2018.

\bibitem[BHW17]{block2017explicit}
Jonathan Block, Julian Holstein, and Zhaoting Wei.
\newblock Explicit homotopy limits of dg-categories and twisted complexes.
\newblock {\em Homology Homotopy Appl.}, 19(2):343--371, 2017.

\bibitem[BK91]{bondal1990enhanced}
A.~I. Bondal and M.~M. Kapranov.
\newblock Enhanced triangulated categories.
\newblock {\em Mathematics of the {USSR}-Sbornik}, 70(1):93--107, 1991.

\bibitem[GJ09]{goerss2009simplicial}
Paul~G. Goerss and John~F. Jardine.
\newblock {\em Simplicial homotopy theory}.
\newblock Modern Birkh\"auser Classics. Birkh\"auser Verlag, Basel, 2009.
\newblock Reprint of the 1999 edition.

\bibitem[Lyu03]{lyubashenko2003category}
Volodymyr Lyubashenko.
\newblock Category of {$A_\infty$}-categories.
\newblock {\em Homology Homotopy Appl.}, 5(1):1--48, 2003.

\bibitem[May92]{may1992simplicial}
J.~Peter May.
\newblock {\em Simplicial objects in algebraic topology}.
\newblock Chicago Lectures in Mathematics. University of Chicago Press,
  Chicago, IL, 1992.
\newblock Reprint of the 1967 original.

\bibitem[OTT81a]{o1981hirzebruch}
Nigel~R. O'Brian, Domingo Toledo, and Yue Lin~L. Tong.
\newblock Hirzebruch-{R}iemann-{R}och for coherent sheaves.
\newblock {\em Amer. J. Math.}, 103(2):253--271, 1981.

\bibitem[OTT81b]{o1981trace}
Nigel~R. O'Brian, Domingo Toledo, and Yue Lin~L. Tong.
\newblock The trace map and characteristic classes for coherent sheaves.
\newblock {\em Amer. J. Math.}, 103(2):225--252, 1981.

\bibitem[OTT85]{o1985grothendieck}
Nigel~R. O'Brian, Domingo Toledo, and Yue Lin~L. Tong.
\newblock A {G}rothendieck-{R}iemann-{R}och formula for maps of complex
  manifolds.
\newblock {\em Math. Ann.}, 271(4):493--526, 1985.

\bibitem[Sei08]{seidel2008fukaya}
Paul Seidel.
\newblock {\em Fukaya categories and {P}icard-{L}efschetz theory}.
\newblock Zurich Lectures in Advanced Mathematics. European Mathematical
  Society (EMS), Z\"{u}rich, 2008.

\bibitem[Tsy18]{tsygan2013microlocal}
Boris Tsygan.
\newblock A microlocal category associated to a symplectic manifold.
\newblock In Michael Hitrik, Dmitry Tamarkin, Boris Tsygan, and Steve Zelditch,
  editors, {\em Algebraic and Analytic Microlocal Analysis}, pages 225--337,
  Cham, 2018. Springer International Publishing.

\bibitem[TT78]{toledo1978duality}
Domingo Toledo and Yue Lin~L. Tong.
\newblock Duality and intersection theory in complex manifolds. {I}.
\newblock {\em Math. Ann.}, 237(1):41--77, 1978.

\bibitem[Wei16]{wei2016twisted}
Zhaoting Wei.
\newblock Twisted complexes on a ringed space as a dg-enhancement of the
  derived category of perfect complexes.
\newblock {\em Eur. J. Math.}, 2(3):716--759, 2016.

\bibitem[Wei18]{wei2018descent}
Zhaoting Wei.
\newblock Descent of dg cohesive modules for open covers on compact complex
  manifolds.
\newblock {\em arXiv preprint arXiv:1804.00993}, 2018.

\bibitem[Wei19]{wei2019recurrent}
Zhaoting Wei.
\newblock A recurrent formula of ${A}_{\infty}$-quasi inverses of dg-natural
  transformations between dg-lifts of derived functors.
\newblock {\em arXiv preprint arXiv:1903.01639}, 2019.

\end{thebibliography}

\bibliographystyle{alpha}

\end{document}